\documentclass[11pt,reqno]{amsart}
\usepackage{amsthm,amssymb,amsmath}
\usepackage[foot]{amsaddr}

\usepackage[T1]{fontenc}

\newtheorem{theorem}{Theorem}
\newtheorem*{theorem*}{Theorem}

\newtheorem{proposition}{Proposition}

\theoremstyle{definition}

\newtheorem*{definition*}{\sc Definition}

\newtheorem{remark}{\sc Remark}
\newtheorem*{remarks}{\sc Remarks}

\newtheorem*{example*}{\bf Example}

\newcommand{\loc}{{\rm loc}}

\newcommand{\clos}{{\rm clos}}

\expandafter\def\expandafter\normalsize\expandafter{%
    \normalsize
    \setlength\abovedisplayshortskip{8pt}
    \setlength\belowdisplayshortskip{8pt}
}

\begin{document}

\title[$\mathcal W^{\alpha, p}$ and $C^{0,\gamma}$ regularity of solutions to $(\mu - \Delta + b \cdot \nabla)u=f$]{$\mathcal W^{\alpha, p}$ and $C^{0,\gamma}$ regularity of solutions to $(\mu - \Delta + b \cdot \nabla)u=f$ with form-bounded vector fields}

\author{Damir Kinzebulatov}

\address{Universit\'{e} Laval, D\'{e}partement de math\'{e}matiques et de statistique, 1045 av.\,de la M\'{e}decine, Qu\'{e}bec, QC, G1V 0A6, Canada}

\email{damir.kinzebulatov@mat.ulaval.ca}

\thanks{The research is supported in part by the Natural Sciences and Engineering Research Council of Canada}



\begin{abstract}
We consider the operator
$-\Delta +b \cdot \nabla$
with $b:\mathbb R^d \rightarrow \mathbb R^d$ ($d \geq 3$) in the class of form-bounded vector fields
(containing vector fields having critical-order singularities), 
and
characterize quantitative dependence of the $\mathcal W^{1+\frac{2}{q},p}$ ($2 \leq p < q$) and the $C^{0,\gamma}$ regularity of solutions to the corresponding elliptic equation in $L^p$ on the value of the form-bound of $b$.
\end{abstract}

\maketitle

Let $d \geq 3$. Consider the formal differential expression 
\begin{equation}
\label{op}
-\Delta + b \cdot \nabla, \quad b:\mathbb R^d \rightarrow \mathbb R^d,
\end{equation}
with $b$ in the class  of form-bounded vector fields $\mathbf{F}_\delta$, $\delta>0$, i.e.\,$|b| \in L^2_{\loc} \equiv L^2_{\loc}(\mathbb R^d,\mathcal L^d)$ and 
there exists a constant $\lambda=\lambda_\delta>0$ such that 
$$
\||b|(\lambda -\Delta)^{-\frac{1}{2}}\|_{2\rightarrow 2}\leq \sqrt{\delta}
$$
(see examples below). It has been established in \cite{KS} that if $\delta<1$, then for every $p \in [2,2/\sqrt{\delta}[$ \eqref{op} has an operator realization $\Lambda_p(b)$ on $L^p$ as the generator  of a positivity preserving, $L^\infty$ contraction, quasi contraction $C_0$ semigroup $e^{-t\Lambda_p(b)}$ such that
$D(\Lambda_p(b)) \subset W^{1,p} \cap W^{1,\frac{pd}{d-2}}.$ Moreover, there exist constants $\mu_1 \equiv \mu_1(d,p,\delta)>0$ and $K_i=K_i(d,p,\delta)>0$, $i=1,2$, such that $u:=(\mu+\Lambda_p(b))^{-1}f$, $f \in L^p$ satisfies for all $\mu>\mu_1$
\begin{equation*}
\|\nabla u\|_p \leq K_1(\mu-\mu_1)^{-\frac{1}{2}}\|f\|_p,  \qquad
\|\nabla|\nabla u|^\frac{p}{2}\|_2^{\frac{2}{p}} \leq K_2(\mu-\mu_1)^{\frac{1}{p}-\frac{1}{2}}\|f\|_p.
\end{equation*}
In particular, if $\delta<1 \wedge \big(\frac{2}{d-2}\bigr)^2$, there exists $p>2 \vee (d-2)$ such that $u \in C^{0,\gamma}$, $\gamma=1-\frac{d-2}{p}$.


The next theorem improves on the regularity of $u$ under the same constraints on $\delta$:

\begin{theorem}[Main result]
\label{thm1} Let $d \geq 3$. Assume that $b \in \mathbf{F}_\delta$, $\delta<1$. Then for every $p \in \big[2,\frac{2}{\sqrt{\delta}}[$ the formal differential expression $-\Delta + b \cdot \nabla$ has an operator realization $\Lambda_p(b)$ on $L^p$ as the generator  of a positivity preserving, $L^\infty$ contraction, quasi contraction $C_0$ semigroup $e^{-t\Lambda_p(b)}$ such that:

{\rm(\textit{i})} The resolvent admits the representation
$$
\big(\mu + \Lambda_p(b)\big)^{-1}=\Theta(\mu,b), \quad \mu > \mu_0,
$$
for a $\mu_0 \equiv \mu_0(d,p,\delta)>0$,
where
$$
\Theta(\mu,b):=(\mu -\Delta )^{-1} - Q_p (1+T_p)^{-1}G_p, 
$$
the operators $Q_p, G_p, T_p \in \mathcal B(L^p)$,
$
\|G_p\|_{p \rightarrow p} \leq C_1 \mu^{-\frac{1}{2}+\frac{1}{p}}$, $\|Q_p\|_{p \rightarrow p} \leq C_2 \mu^{-\frac{1}{2}-\frac{1}{p}}$, $ \|T_p\|_{p \rightarrow p} \leq c_{\delta,p}< 1,
$ where $c_{\delta,p}:=\bigl(\frac{p}{2}\delta + \frac{p-2}{2}\sqrt{\delta}\bigr)^{\frac{1}{p}}\bigl(p-1-(p-1)\frac{p-2}{2}\sqrt{\delta} - \frac{p(p-2)}{4}\delta\bigr)^{-\frac{1}{p}}$,
$$
G_p:=b^{\frac{2}{p}} \cdot \nabla (\mu -\Delta)^{-1}, \quad b^{\frac{2}{p}}:=|b|^{\frac{2}{p}-1}b,
$$
and $Q_p$, $T_p$ are the extensions by continuity of densely defined {\rm(}on $\mathcal E:=\bigcup_{\varepsilon>0}e^{-\varepsilon|b|}L^p${\rm)} operators
$$
Q_p \upharpoonright {\mathcal E}:=(\mu -\Delta)^{-1}|b|^{1-\frac{2}{p}}, \quad T_p \upharpoonright  {\mathcal E}:=b^{\frac{2}{p}}\cdot \nabla(\mu - \Delta)^{-1}|b|^{1-\frac{2}{p}}.
$$

{\rm (\textit{ii})} For each $2 \leq r<p<q<\infty$ and $\mu > \mu_0$, define  
$$
G_p(r):=b^{\frac{2}{p}} \cdot \nabla (\mu -\Delta )^{-\frac{1}{2}-\frac{1}{r}} \in \mathcal B(L^p), \qquad 
Q_p(q):=(\mu -\Delta )^{-\frac{1}{2}+\frac{1}{q}}|b|^{1-\frac{2}{p}} \quad
\text{ on } \mathcal E.
$$
The extension of $Q_p(q)$ by continuity  we denote again by $Q_{p}(q)$.
Then $Q_{p}(q)\in \mathcal B(L^p)$ and
$$
\Theta_p(\mu, b)= (\mu - \Delta)^{-1} - (\mu - \Delta)^{-\frac{1}{2}-\frac{1}{q}} Q_{p}(q) (1 + T_p)^{-1} G_{p}(r) (\mu - \Delta)^{-\frac{1}{2}+\frac{1}{r}}, \qquad \mu>\mu_0.\\
$$
Thus,
\begin{equation}
\label{reg}
\tag{$\star$}
\big(\mu + \Lambda_p(b)\big)^{-1} \in \mathcal B\bigl(\,\mathcal W^{-1+\frac{2}{r},p}, \; \mathcal W^{1+\frac{2}{q},p}\,\bigr)
\end{equation}
{\rm (}$\mathcal W^{\alpha,p}$ is the Bessel potential space{\rm)}.

\smallskip

{\rm (\textit{iii})}\; By {\rm(\textit{i})} and {\rm(\textit{ii})},
$D\bigl(\Lambda_p(b)\bigr)\subset \mathcal W^{1+\frac{2}{q},p}$ {\rm($q>p$)}. In particular, by the Sobolev Embedding Theorem, for $d \geq 4$, if $\delta<\big(\frac{2}{d-2}\bigr)^2$ then there exists $p>d-2$ such that $D\bigl(\Lambda_p(b)\bigr) \subset C^{0,\gamma}$, $\gamma<1-\frac{d-2}{p}$. {\rm(}For $d=3$ the corresponding inclusion can be improved, see remarks below.{\rm)}

\medskip

{\rm (\textit{iv})} \; $e^{-t \Lambda_p (b_n)} \rightarrow e^{-t \Lambda_p(b)}$ \text{ strongly in $L^p$ \;\;locally uniformly in} $t \geq 0,$

\medskip
\noindent 
where
$b_n := e^{\epsilon_n\Delta} (\mathbf 1_n b)$,  $\epsilon_n \downarrow 0$, $n \geq 1$,  $\mathbf 1_n$ is the indicator of  $\{x \in \mathbb R^d \mid  \; |x| \leq n,  |b(x)| \leq n \}$,
and
 $\Lambda_{p}(b_n):=-\Delta + b_n\cdot \nabla$, $D(\Lambda_p(b_n))=\mathcal W^{2,p}$.

\end{theorem}

\begin{remarks}
1.~For $d=3$, by the Miyadera Perturbation Theorem, the assumption $b \in \mathbf F_\delta$, $ \delta < 1$ implies that $-\Lambda_2(b) = \Delta - b\cdot \nabla$ of domain $W^{2,2}$ is the generator of a $C_0$ semigroup in $L^2$, and hence, for $\mu > \lambda\delta$, $(\mu + \Lambda_2(b))^{-1} : L^2 \rightarrow W^{1,6}.$ In particular, $D(\Lambda_2(b)) \subset C^{0,\gamma}$ with $\gamma=\frac{1}{2}$.

2.~The class $\mathbf{F}_\delta$ contains a sub-critical class $[L^d + L^\infty]^d$ (with arbitrarily small form-bound $\delta$) as well as  vector fields having critical-order singularities, e.g.\,in the weak $L^{d}$ class or the Campanato-Morrey class etc. See e.g.\,\cite[sect.\,4]{KiS}.

3.~We say that $b: \mathbb{R}^d \rightarrow  \mathbb{R}^d$ belongs to $\mathbf{F}_\delta^{\scriptscriptstyle 1/2}$, the class of \textit{weakly} form-bounded vector fields, and write $b \in \mathbf{F}_\delta^{\scriptscriptstyle 1/2}$, if 
$|b| \in L^1_{\loc}$ and there exists $\lambda = \lambda_\delta > 0$ such that
$$
\| |b|^\frac{1}{2} (\lambda - \Delta)^{-\frac{1}{4}} \|_{2 \rightarrow 2} \leq \sqrt{\delta}.
$$
In \cite[Theorem 1.3]{Ki}, \cite[Theorem 4.3]{KiS}, we have constructed an operator realization $\Lambda_p(b)$ of $-\Delta + b \cdot \nabla$, $b \in \mathbf{F}_\delta^{\scriptscriptstyle 1/2}$, $m_d\delta<1$, $m_d:= \pi^\frac{1}{2}(2 e)^{-\frac{1}{2}} d^\frac{d}{2} (d-1)^{-\frac{d-1}{2}}$ as the generator of a positivity preserving, $L^\infty$ contraction, holomorphic semigroup on $L^p$, $p \in ]p_-,p_+[$,  $
 p_\mp:= \frac{2}{1 \pm \sqrt{1-m_d \delta}}$, such that for all $1 \leq r<p<q$
\begin{equation}
\tag{$\star\star$}
\label{reg2}
\big(\zeta + \Lambda_p(b)\big)^{-1} \in \mathcal B(\mathcal W^{-1+\frac{1}{r},p},\mathcal W^{1+\frac{1}{q},p})
\end{equation}
(cf.\,\eqref{reg}). In particular, if $m_d\delta<4\frac{d-2}{(d-1)^2}$, then there exists a $p>d-1$ such that $D(\Lambda_p(b))\subset C^{0,\gamma}$, $\gamma<1-\frac{d-1}{p}$.

(Despite the inclusion $\mathbf{F}_{\delta^2} \subsetneq \mathbf{F}_\delta^{\scriptscriptstyle 1/2}$, see \cite[sect.\,4]{KiS}, these two classes should be viewed as essentially incomparable, for the corresponding regularity results hold under different assumptions on the value of $\delta$.)

The proof of Theorem \ref{thm1} follows the Hille-Trotter approach of \cite{Ki}, \cite{KiS}.
However, the proof of the crucial estimates in Proposition \ref{prop1} below is based on \cite{KS}. 

4.~For $|b| \in L^{d,\infty}$, one can extract additional information about the regularity of $D(\Lambda_p(b))$  arguing as in remark 4 in \cite[sect.\,4.4]{KiS}.

5.~Let $C_\infty:=\{f \in C(\mathbb R^d): \lim_{x \rightarrow \infty}f(x)=0\}$ (with the $\sup$-norm). 
 Theorem \ref{thm1} allows to construct the generator $\Lambda_{C_\infty}(b)$ of an associated with $-\Delta + b \cdot\nabla$, $b \in \mathbf{F}_\delta$, $\delta<1 \wedge \big(\frac{2}{d-2}\bigr)^2$ Feller semigroup as $(\mu+\Lambda_{C_\infty}(b))^{-1}:=\bigl(\Theta_p(\mu,b) \upharpoonright L^p \cap C_\infty \bigr)_{C_\infty \rightarrow C_\infty}^{\clos}$, $p>2 \vee (d-2)$, by repeating \cite[proof of Theorem 4.4]{KiS} (for $-\Delta +b \cdot \nabla$, $b \in \mathbf{F}_\delta^{\scriptscriptstyle 1/2}$). Thus, we restore the result of \cite[Theorem 2]{KS}. This proof doesn't require the Moser-type iteration procedure $L^p \rightarrow L^\infty$ of \cite{KS}.

6.~The proof of Theorem \ref{thm1} extends directly to the operator studied in \cite{KiS2}:
$$-\nabla \cdot a \cdot \nabla + b \cdot \nabla \equiv -\sum_{i,j=1}^d  \nabla_i a_{ij}(x) \nabla_j + \sum_{k=1}^d b_k(x) \nabla_k, \qquad b \in \mathbf{F}_\delta,
$$
with
$$
a=I+c\,\mathsf{f} \otimes \mathsf{f} \qquad \text{ where } \quad c>-1, \quad \mathsf{f} \in \bigl[L^\infty \cap W_{\loc}^{1,2}\bigr]^d, \quad \|\mathsf{f}\|_\infty=1,
$$
$$
\nabla_i \mathsf{f} \in \mathbf{F}_{\eta^i},\;\;i=1,2,\dots,d, \quad
 \eta:=\sum_{i=1}^d\eta^i
$$
for all $p \geq 2$, $c$, $\eta$ and $\delta$ satisfying the assumptions of \cite[Theorem 2]{KiS2}.
(More generally, with
$a=I+\sum_{\ell=1}^\infty c_\ell\,\mathsf{f}_\ell \otimes \mathsf{f}_\ell$, $\mathsf{f}_\ell \in \bigl[L^\infty \cap W^{1,2}_{\loc}\bigr]^d$, $\|\mathsf{f}_\ell\|_\infty=1$
such that
$
\sum_{c_\ell<0} c_\ell>-1$, $\sum_{c_\ell>0} c_\ell<\infty
$ and
$
\nabla_i \mathsf{f}_\ell \in \mathbf{F}_{\eta^i_{\ell}}$, $i=1,2,\dots,d$, 
for appropriate $\eta^i_{\ell}>0$.)

\end{remarks}

\noindent\textbf{Acknowledgements.} I would like to express my gratitude to Yu.\,A.\,Semenov for helpful discussions.

\section{Proof of Theorem \ref{thm1}}

The following proposition is a new element in the Hille-Trotter approach of \cite{Ki}, \cite{KiS}.

\begin{proposition} 
\label{prop1}
{\rm(\textit{j})} Set 
$
G_p=b^{\frac{2}{p}} \cdot \nabla (\mu -\Delta )^{-1}$, $
Q_p=(\mu -\Delta )^{-1}|b|^{1-\frac{2}{p}}$, $T_p=b^{\frac{2}{p}}\cdot\nabla(\mu -\Delta)^{-1}|b|^{1-\frac{2}{p}}.
$
$Q_p$, $T_p$ are densely defined {\rm(}on $\mathcal E${\rm)} operators. Then there exists $\mu_0 \equiv \mu_0(d,p,\delta)>0$ such that
 $$
\|G_p\|_{p \rightarrow p} \leq C_1 \mu^{-\frac{1}{2}+\frac{1}{p}}, \quad \|Q_p\|_{p \rightarrow p} \leq C_2 \mu^{-\frac{1}{2}-\frac{1}{p}}, \quad \|T_p\|_{p \rightarrow p} \leq c_{\delta,p} < 1, \qquad \mu>\mu_0,
$$
where $c_{\delta,p}:=\biggl(\frac{p}{2}\delta + \frac{p-2}{2}\sqrt{\delta}\biggr)^{\frac{1}{p}}\biggl(p-1-(p-1)\frac{p-2}{2}\sqrt{\delta} - \frac{p(p-2)}{4}\delta\biggr)^{-\frac{1}{p}}$.

\smallskip

{\rm(\textit{jj})} Set 
$
G_p(r)=b^{\frac{2}{p}} \cdot \nabla (\mu -\Delta )^{-\frac{1}{2}-\frac{1}{r}}$, 
$Q_p(q)=(\mu -\Delta )^{-\frac{1}{2}+\frac{1}{q}}|b|^{1-\frac{2}{p}}, 
$
where $2 \leq r<p<q<\infty$. $Q_p(q)$ is a densely defined {\rm(}on $\mathcal E${\rm)} operator. Then for $\mu>\mu_0$
$$
\|G_p(r)\|_{p \rightarrow p} \leq K_{1,r}, \qquad \|Q_p(q)\|_{p \rightarrow p} \leq K_{2,q}.
$$
The extension of $Q_p(q)$ by continuity we denote again by $Q_p(q)$.

\end{proposition}

\begin{proof}
In what follows, we use notation
$$
\langle h\rangle:=\int_{\mathbb R^d} h(x)d\mathcal L^d, \quad \langle h,g\rangle:=\langle h\bar{g}\rangle.
$$

It suffices to consider the case $p>2$.

(\textit{j}) (\textbf{a}) Set $u:=(\mu-\Delta)^{-1}|b|^{1-\frac{2}{p}}f$, $0 \leq f \in L^p$. Then
\begin{align*}
&\|T_p f\|_p^p =\|b^{\frac{2}{p}} \nabla u\|_p^p = \langle |b|^2 |\nabla u|^p\rangle \notag \\
& = \||b|(\lambda-\Delta)^{-\frac{1}{2}} (\lambda-\Delta)^{\frac{1}{2}}|\nabla u|^{\frac{p}{2}}\|_2^2 \notag  \qquad (\lambda=\lambda_\delta) \notag  \\
& \leq \||b|(\lambda-\Delta)^{-\frac{1}{2}}\|_{2 \rightarrow 2}^2 \|(\lambda-\Delta)^{\frac{1}{2}}|\nabla u|^{\frac{p}{2}}\|_2^2 \notag  \\
& = \delta \|(\lambda-\Delta)^{\frac{1}{2}}|\nabla u|^{\frac{p}{2}}\|_2^2= \delta \bigl( \lambda\|\nabla u\|_p^p + \|\nabla |\nabla u|^{\frac{p}{2}}\|_2^2 \bigr). \notag
\end{align*}
It remains to prove the principal inequality
\[
\delta \bigl( \lambda\|\nabla u\|_p^p + \|\nabla |\nabla u|^{\frac{p}{2}}\|_2^2 \bigr) \notag  
\leq c^p_{\delta,p} \|f\|_p^p \tag{$\ast$} \label{ineq1}, 
\]
and conclude that $\|T_p\|_{p \rightarrow p} \leq c_{\delta, p}$. 

First, we prove an a priori variant of \eqref{ineq1}, i.e.\,for $u:=(\mu-\Delta)^{-1}|b|^{1-\frac{2}{p}}f$ with $b=b_n$. Since our assumptions on $\delta$ involve only strict inequalities, we may assume, upon selecting appropriate $\varepsilon_n \downarrow 0$, that $b_n \in \mathbf{F}_\delta$ with the same $\lambda=\lambda_\delta$ for all $n$.

 Set 
$$
w:=\nabla u, \quad I_q:=\sum_{r=1}^d\langle (\nabla_r w)^2 |w|^{p-2} \rangle, \quad
J_q:=\langle (\nabla |w|)^2 |w|^{p-2}\rangle.
$$
We multiply $(\mu-\Delta)u=|b|^{1-\frac{2}{p}}f$ by $\phi:=- \nabla \cdot (w|w|^{p-2})$ and integrate by parts to obtain
\begin{equation}
\label{main_id}
\mu \|w\|_p^p + I_p + (p-2)J_p = \langle |b|^{1-\frac{2}{p}}f, - \nabla \cdot (w|w|^{p-2})\rangle,
\end{equation}
where
\begin{align*}
&\langle |b|^{1-\frac{2}{p}}f, - \nabla \cdot (w|w|^{p-2})\rangle = \langle |b|^{1-\frac{2}{p}}f, (- \Delta u) |w|^{p-2} - (p-2) |w|^{p-3} w \cdot \nabla |w|\rangle \\
& (\text{use the equation } -\Delta u = - \mu u + |b|^{1-\frac{2}{p}}f) \\
&=\langle |b|^{1-\frac{2}{p}}f, \bigl( - \mu u + |b|^{1-\frac{2}{p}}f\bigr) |w|^{p-2}\rangle  - (p-2) \langle |b|^{1-\frac{2}{p}}f, |w|^{p-3} w \cdot \nabla |w|\rangle.
\end{align*}
We have

1) $\langle |b|^{1-\frac{2}{p}}f, (- \mu u)|w|^{p-2} \rangle \leq 0$,

2) $|\langle |b|^{1-\frac{2}{p}}f, |w|^{p-3} w \cdot \nabla |w|\rangle| \leq \alpha J_p+\frac{1}{4\alpha} N_p\;\;( \alpha>0$), where $N_p:=\langle |b|^{1-\frac{2}{p}}f, |b|^{1-\frac{2}{p}}f |w|^{p-2} \rangle$,

\noindent so, the RHS of \eqref{main_id} $\leq  (p-2) \alpha J_p +  \big(1+\frac{p-2}{4\alpha}\big)N_p,
$
where, in turn,
\begin{align*}
N_p& \leq \langle |b|^{2}|w|^p\rangle^{\frac{p-2}{p}} \langle f^{p}\rangle^{\frac{2}{p}} \\
&\leq \frac{p-2}{p} \langle |b|^{2}|w|^p\rangle + \frac{2}{p} \|f\|_p^p \qquad \text{(use $b \in \mathbf{F}_\delta$ $\Leftrightarrow$ $\|b\varphi\|_2^2 \leq \delta \|\nabla \varphi\|_2^2 + \lambda\delta\|\varphi\|_2^2$, $\varphi \in W^{1,2}$)} \\
& \leq \frac{p-2}{p} \bigg(\frac{p^2}{4}\delta J_q + \lambda\delta \|w\|_p^p \bigg) + \frac{2}{p} \|f\|_p^p.
\end{align*}
Thus, applying $I_q \geq J_q$ in the LHS of \eqref{main_id}, we obtain
$$
\bigl(\mu - c_0 \bigr)\|w\|_p^p + \biggl[p-1-(p-2)\left(\alpha + \frac{1}{4\alpha}\frac{p(p-2)}{4}\delta\right) - \frac{p(p-2)}{4}\delta\biggr]\frac{4}{p^2}\|\nabla |\nabla u|^{\frac{p}{2}}\|_2^2 \leq \left(1+\frac{p-2}{4\alpha}\right)\frac{2}{p}\|f\|_p^p,
$$
where $c_0=\frac{p-2}{p}\lambda\delta \big(1+\frac{p-2}{4\alpha}\big)$. It is now clear that one can find a sufficiently large $\mu_0 \equiv \mu_0(d,p,\delta)>0$ so that, for all $\mu > \mu_0$, \eqref{ineq1} (with $b=b_n$) holds  with
\begin{align*}
 c_{\delta,p}^p&=\delta \frac{p^2}{4}\frac{\left(1+\frac{p-2}{4\alpha}\right)\frac{2}{p}}{p-1-(p-2)\left(\alpha + \frac{1}{4\alpha}\frac{p(p-2)}{4}\delta\right) - \frac{p(p-2)}{4}\delta} \qquad \text{$\bigl($we  select $\alpha=\frac{p}{4}\sqrt{\delta}\bigr)$} \\
& = \frac{\frac{p}{2}\delta + \frac{p-2}{2}\sqrt{\delta}}{p-1-(p-1)\frac{p-2}{2}\sqrt{\delta} - \frac{p(p-2)}{4}\delta},
\end{align*}
as claimed.
Finally, we pass to the limit $n \rightarrow \infty$ using Fatou's Lemma. The proof of \eqref{ineq1} is completed.

\begin{remark}
It is seen that $\sqrt{\delta}<\frac{2}{p}\Rightarrow c_{\delta, p}<1$. We also note that the above choice of $\alpha$ is the best possible.
\end{remark}

\smallskip

(\textbf{b}) Set $u=(\mu-\Delta)^{-1}f$, $0 \leq f \in L^p$. Then
\begin{align*}
&\|G_p f\|_p^p=\|b^{\frac{2}{p}} \cdot \nabla u\|_p^p \\
& \text{(we argue as in (\textbf{a}))} \\
&\leq \delta \big(\lambda\|\nabla u\|_p^p + \|\nabla |\nabla u|^{\frac{p}{2}}\|_2^2 \big),
\end{align*}
where, clearly, $\|\nabla u\|^p_p \leq \mu^{-\frac{p}{2}}\|f\|_p^p$. In turn, arguing as in (\textbf{a}), we arrive at $\mu\|w\|_p^p + I_p + (p-2)J_p = \langle f,-\nabla \cdot (w|w|^{p-2})$ ($w=\nabla u$),
$$
\mu\|w\|_p^p + (p-1)J_p \leq \langle f^2,|w|^{p-2}\rangle + (p-2)\langle f, |w|^{p-3}w \cdot \nabla |w|\rangle),
$$
$$
\mu\|w\|_p^p + (p-1)J_p \leq \langle f^2,|w|^{p-2}\rangle + (p-2)\bigl(\varepsilon J_p + \frac{1}{4\varepsilon}\langle f^2,|w|^{p-2}\rangle  \bigr), \quad \varepsilon>0.
$$
Selecting $\varepsilon$ sufficiently small, we obtain
$$
J_p \leq C_0\|w\|_p^{p-2}\|f\|_p^2.
$$
Now,  applying $\|w\|_p \leq \mu^{-\frac{1}{2}}\|f\|_p$, we arrive at $\|\nabla |\nabla u|^{\frac{p}{2}}\|_2^2 \leq C\mu^{-\frac{p}{2}+1}\|f\|_p^p$. 
Hence,
$\|G_pf\|_{p} \leq C_1 \mu^{-\frac{1}{2}+\frac{1}{p}}\|f\|_p$ for all $\mu>\mu_0$.

\smallskip

(\textbf{c}) Set $u=(\mu-\Delta)^{-1}|b|^{1-\frac{2}{p}}f\;(=Q_p f)$, $0 \leq f \in L^p$. Then, multiplying  $(\mu  - \Delta )u = |b|^{1-\frac{2}{p}}f$ by $u^{p-1}$, we obtain
$$
\mu \|u\|_p^p + \frac{4(p-1)}{p^2} \|\nabla u^{\frac{p}{2}}\|_2^2 = \langle |b|^{1-\frac{2}{p}}f, u^{p-1} \rangle,
$$
where we estimate the RHS using Young's inequality:
\begin{align*}
\langle |b|^{1-\frac{2}{p}} u^{\frac{p}{2}-1}, f u^{\frac{p}{2}} \rangle \leq \varepsilon^{\frac{2p}{p-2}}\frac{p-2}{2p} \langle |b|^2 u^p \rangle + \varepsilon^{-\frac{2p}{p+2}}\frac{p+2}{2p} \langle f^{\frac{2p}{p+2}} u^{\frac{p^2}{p+2}}\rangle \quad \varepsilon>0.
\end{align*}
Using $b \in \mathbf{F}_\delta$ and selecting $\varepsilon>0$ sufficiently small, we obtain that for any $\mu_1>0$ there exists $C>0$ such that
$$
(\mu-\mu_1) \|u\|_p^p \leq C\langle f^{\frac{2p}{p+2}} u^{\frac{p^2}{p+2}}\rangle, \qquad \mu>\mu_1.
$$
Therefore,
$
(\mu-\mu_1) \|u\|_p^p \leq C\langle f^p \rangle^{\frac{2}{p+2}}\langle  u^{p}\rangle^{\frac{p}{p+2}}
$, so
$\|u\|_p \leq C_2 \mu^{-\frac{1}{2}-\frac{1}{p}}\|f\|_p
$. 
The proof of (\textit{j}) is completed.

\medskip

(\textit{jj})  Below we use the following formula: For every $0<\alpha<1$, $\mu > 0$,
\begin{equation*}
(\mu-\Delta)^{-\alpha}=\frac{\sin \pi \alpha}{\pi}\int_0^\infty t^{-\alpha}(t+\mu-\Delta)^{-1}dt.
\end{equation*}

We have 
\begin{align*}
\|Q_{p}(q)f\|_p & \leq ~\|(\mu-\Delta)^{-\frac{1}{2}+\frac{1}{q}}|b|^{1-\frac{2}{p}}|f|\|_p \\
& \leq 
k_{q}\int_0^\infty t^{-\frac{1}{2}+\frac{1}{q}}\|(t+\mu-\Delta)^{-1}|b|^{1-\frac{2}{p}}|f|\|_p dt  \\
& \text{(we use (\textbf{c}))} \\
& \leq ~
k_{q}C_{2} \int_0^\infty t^{-\frac{1}{2}+\frac{1}{q}}(t+\mu)^{-\frac{1}{2}-\frac{1}{p}}dt \;\|f\|_p =K_{2,q}\|f\|_p,\quad f \in \mathcal E,
\end{align*}
where, clearly, $K_{2,q}<\infty$ due to $q>p$.

\smallskip

It suffices to consider the case $r>2$. We have
\begin{align*}
\|G_{p}(r)f\|_p &\leq 
k_{r}\int_0^\infty t^{-\frac{1}{2}-\frac{1}{r}}\|b^{\frac{2}{p}}\cdot \nabla (t+\mu-\Delta)^{-1} f\|_pdt \;\;
\\ 
&(\text{we use (\textbf{b})})\\ 
&\leq 
k_{r}C_{1} \int_0^\infty t^{-\frac{1}{2}-\frac{1}{r}}(t+\mu)^{-\frac{1}{2}+\frac{1}{p}}dt\;\|f\|_p = K_{1,r}\|f\|_p,  \quad f \in \mathcal E,
\end{align*}
where, clearly, $K_{1,r}<\infty$ due to $r<p$. 

The proof of (\textit{jj}) is completed.
\end{proof}

\begin{remark}
Proposition \ref{prop1} is valid for $b_n$, $n=1,2,\dots$, with the same constants.
\end{remark}

\begin{proposition}
\label{lem_15}
The operator-valued function $\Theta_{p}(\mu,b_n)$ is a pseudo-resolvent on $\mu > \mu_0$, i.e.
 \begin{equation*}
\Theta_p(\mu,b_n) - \Theta_p(\nu,b_n) = (\nu - \mu) \Theta_p(\mu,b_n)\Theta_p(\nu,b_n), \quad \mu, \nu > \mu_0.
\end{equation*}
\end{proposition}
\begin{proof}
The proof repeats  \cite[proof of Prop.\,2.4]{Ki}.
\end{proof}

\begin{proposition}
\label{lem_18}
For every $n=1,2,\dots$,
$$\mu \Theta_{p}(\mu,b_n) \rightarrow 1 \text{ strongly in $L^p$ as } \mu\uparrow \infty \quad  (\text{uniformly in $n$}).$$
\end{proposition}
\begin{proof}
The proof repeats \cite[proof of Prop.\,2.5(\textit{ii})]{Ki}.
\end{proof}

\begin{proposition}
\label{lem_20}
We have $\{\mu: \mu > \mu_0\} \subset \rho(-\Lambda_p(b_n))$, the resolvent set of $-\Lambda_p(b_n)$. The operator-valued function $\Theta_{p}(\mu,b_n)$ is the resolvent of $-\Lambda_p(b_n)$:
\begin{equation*}
\Theta_{p}(\mu,b_n)=(\mu+\Lambda_p(b_n))^{-1}, \quad \mu > \mu_0.
\end{equation*}
\end{proposition}
\begin{proof}
The proof repeats \cite[proof of Prop.\,2.6]{Ki}.
\end{proof}

\begin{proposition}
\label{lem_contr}
We have, for all $n=1,2,\dots$,
$$
\|(\mu+\Lambda_p(b_n))\|_{p \rightarrow p} \leq (\mu-\mu_0)^{-1}, \quad \mu > \mu_0:=\mu_0 \vee \frac{\lambda\delta}{2(p-1)}. 
$$
\end{proposition}

\begin{proof}
By \cite[Theorem 1]{KS}. 
\end{proof}

\begin{proposition}
\label{lem_30}
For every $\mu > \mu_0$,
\begin{equation*}
\Theta_{p} (\mu,b_n) \rightarrow \Theta_{p} (\mu,b) \text{ strongly in $L^p$}.
\end{equation*}
\end{proposition}
\begin{proof}
The proof repeats \cite[proof of Prop.\,2.8]{Ki}.
\end{proof}

Now, by the Trotter Approximation Theorem \cite[IX.2.5]{Ka}, 
  $\Theta_p(\mu,b)=(\mu+\Lambda_p(b))^{-1}$, $\mu > \mu_0,$
where $\Lambda_p(b)$ is the generator of a quasi contraction $C_0$ semigroup in $L^p$.
(\textit{i}) follows.
 (\textit{ii}) follows from Proposition \ref{prop1}(\textit{jj}).
(\textit{ii}) $\Rightarrow$ (\textit{iii}). 
(\textit{iv}) is Proposition \ref{lem_30}.
The proof of Theorem \ref{thm1} is completed.

\end{document}